\documentclass[11pt,reqno]{amsart}
\pdfoutput=1
\usepackage[]{amsmath,amssymb,amsfonts,latexsym,amsthm,enumerate}
\usepackage{amssymb,bbm,mathrsfs,tikz}
\usepackage{enumerate,graphicx,paralist}
\usepackage[noadjust]{cite}

\numberwithin{equation}{section}
\usepackage[colorlinks=true, pdfstartview=FitV, linkcolor=blue,
  citecolor=blue, urlcolor=blue,pagebackref=false]{hyperref}


\newtheorem{theorem}{Theorem}[section]
\newtheorem*{theorem*}{Theorem}
\newtheorem{lemma}[theorem]{Lemma}

\newtheorem{proposition}[theorem]{Proposition}

\newtheorem{corollary}[theorem]{Corollary}

\theoremstyle{definition}{

\newtheorem{definition}[theorem]{Definition}
\newtheorem*{definition*}{Definition}

\newtheorem*{question*}{Question}
\newtheorem*{example*}{Example}
\newtheorem*{examples*}{Examples}
\newtheorem{remark}[theorem]{Remark}
\newtheorem*{remark*}{Remark}

}

\DeclareMathOperator{\xor}{\triangle}

\newcommand{\abbr}[1]{{\sc{\lowercase{#1}}}}

\def\R{{\mathbb R}}

\def\P{{\mathbb P}}

\def\N{{\mathbb N}}

\newcommand{\cB}{{\mathcal B}}

\newcommand{\cG}{{\mathcal G}}

\newcommand{\sss}{{\mathbb S}}

\newcommand{\cW}{{\mathcal W}}

\newcommand{\cX}{{\mathcal X}}

\newcommand{\Bin}{\operatorname{Bin}}

\renewcommand{\epsilon}{\varepsilon}

\author{Amir Dembo}
\address{Amir Dembo\hfill\break
Department of Mathematics\\ Stanford University\\ Sloan Hall \\
Stanford, CA 94305, USA.}
\email{amir@math.stanford.edu}

\author{Eyal Lubetzky}
\address{Eyal Lubetzky\hfill\break
Courant Institute 
\\ New York University\\
251 Mercer Street\\ New York, NY 10012, USA.}
\email{eyal@courant.nyu.edu}

\title[Large deviations in uniform random graphs]
{A large deviation principle for the \\  Erd\H{o}s--R\'enyi uniform random graph}

\begin{document}

\begin{abstract}\vspace{-0.2cm}
Starting with the large deviation principle (\abbr{LDP})
for the Erd\H{o}s--R\'enyi binomial random graph $\cG(n,p)$ 
(edge indicators are i.i.d.), due to 
Chatterjee and Varadhan (2011), we derive the \abbr{LDP} for the uniform random graph $\cG(n,m)$ (the uniform distribution over graphs with $n$ vertices and $m$ edges), at suitable $m=m_n$. Applying the latter \abbr{ldp} we find that 
tail decays for subgraph counts in $\cG(n,m_n)$ are controlled by variational 
problems, which up to a constant shift, coincide with those
studied by Kenyon et al.\ and Radin et al.\ in the context of 
constrained random graphs, e.g., the edge/triangle model. 
\end{abstract}


{\mbox{}\vspace{-2.2cm}
\maketitle
}
\vspace{-0.8cm}

\section{Introduction}\label{secIntro}
The Erd\H{o}s--R\'enyi \emph{binomial} random graph model $\cG(n,p)$ is the graph on $n$ vertices where each edge is present independently with probability $p$; the \emph{uniform} random graph $\cG(n,m)$ is the uniform distribution over graphs with $n$ vertices and exactly $m$ edges.

Let $\cW$ be the space of all bounded measurable functions
$f:[0,1]^2 \to \R$ that are symmetric ($f(x,y) = f(y,x)$ for all $x,y \in
[0,1]$). Let $\cW_0\subset \cW$ denote all \emph{graphons}, that is, symmetric measurable functions
$[0,1]^2 \to [0,1]$ (these generalize finite graphs; see~\eqref{eq:W_G-def}).  The cut-norm of $W\in\cW$ is given by  
\begin{align*}
\left\|W\right\|_\square &:= \sup_{S, T \subset [0,1]} \bigg|\int_{S \times T}
  W(x,y)  \ dxdy\bigg| = \sup_{u,v \colon [0,1] \to [0,1]} \bigg|\int_{[0,1]^2} W(x,y)u(x)v(y)
  \ dxdy\bigg|\,,
\end{align*}
(by linearity of the integral it suffices to consider $\{0,1\}$-valued $u,v$, hence the equality).
For any measure-preserving map $\sigma \colon [0,1] \to [0,1]$ and $W
\in \cW$, let $W^\sigma \in \cW$ denote the graphon 
$W^\sigma(x,y) = W(\sigma (x), \sigma
(y))$. The cut-distance on $\cW$ is then defined as
\[
\delta_\square (W_1,W_2) := \inf_{\sigma} \left\|W_1 - W_2^\sigma\right\|_\square
\,,
\]
with the infimum taken over all measure-preserving bijections $\sigma$ 
on $[0,1]$. It yields the pseudo-metric space $(\cW_0,
\delta_\square)$, which is elevated into a genuine metric
space $(\widetilde \cW_0,\delta_\square)$
upon taking the quotient w.r.t.\ the equivalence relation $W_1\sim W_2$ iff $\delta_\square(W_1,W_2) = 0$.
In what may be viewed as a topological version of Szemer\'edi's regularity
lemma, Lov\'asz and Szegedy~\cite{LS06} showed that the metric space
 $(\widetilde \cW_0, \delta_\square)$ is compact.
 For a finite simple graph $H=(V(H),E(H))$ with $V(H) = \{1,\ldots,k\}$, its subgraph density in $W\in\cW_0$ is 
\[
t_H(W) := \int_{[0,1]^k} \prod_{(i,j) \in E(H)} W(x_i, x_j) \
dx_1 \cdots dx_k\,,
\]
with the map $W\mapsto t_H(W)$ being Lipschitz-continuous in $(\widetilde\cW_0,\delta_\square)$ 
  (see~\cite[Thm~3.7]{BCLSV08}).  

Define $I_p \colon [0,1] \to \R$ by
\begin{equation}
  \label{eq:Ip}
I_p(x) := \frac{x}2  \log \frac{x}{p} + \frac{1-x}2 \log \frac{1-x}{1-p}\quad\mbox{ for $p \in (0,1)$ and $x \in [0,1]$}\,,
\end{equation}
and extend $I_p$ to $\cW_0$ via $I_p(W) := \int_{[0,1]^2} I_p(W(x,y)) \ dxdy$ for $W \in \cW_0$. As $I_p$ is
convex on $[0,1]$, it is lower-semicontinuous on
$\widetilde \cW_0$ w.r.t.\ the cut-metric topology (\cite[Lem.~2.1]{CV11}).

In the context of the space of graphons $\widetilde\cW_0$, a simple graph $G$ with vertices $\{1, \ldots, n\}$ can be
represented by
\begin{equation}\label{eq:W_G-def}
W_G(x,y) = \begin{cases}
                 1 & \text{if $(\lceil nx\rceil, \lceil ny\rceil )$ is an edge of $G$}, \\
                 0 & \text{otherwise.}
               \end{cases}
\end{equation}
For two graphs $G$ and $H$ let $\hom(H, G)$ count the number of
homomorphisms from $H$ to $G$ (i.e., maps $V(H) \to V(G)$ that carry
edges to edges). Let
\[
t_H(G) :=|V(G)|^{-|V(H)|}|\hom(H,G)| = t_H(W_G)\,.
\]
A sequence of graphs $\{G_n\}_{n \geq 1}$ is said to converge if
the sequence of subgraph densities $t_H(G_n)$ converges for every fixed finite
simple graph $H$. It was shown in~\cite{LS06} that for any
such convergent graph sequence there is a limit object $W \in \widetilde \cW_0$ such
that $t_H(G_n) \to t_H(W)$ for every fixed $H$. Conversely, any
$W \in \widetilde \cW_0$ arises as a limit of a convergent graph sequence.
It was shown in \cite{BCLSV08} that a sequence of graphs
$\{G_n\}_{n \geq 1}$ converges if
and only if the sequence of graphons $W_{G_n} \in \cW_0$ converges in
$\cW_0$ w.r.t.\ $\delta_\square$

A random graph $G_n \sim \cG(n,p)$ corresponds to
a random point $W_{G_n} \in  \widetilde \cW_0$---inducing a
probability distribution $\P(G_n\in\cdot)$ on $\widetilde \cW_0$ supported on a
finite set of points ($n$-vertex graphs)---and $G_n\to W$ for the constant graphon $W\equiv p$ a.s.\ for every fixed $0<p<1$.
Chatterjee and Varadhan~\cite{CV11} showed that, for $0<p<1$ fixed, the random graph $\cG(n,p)$ obeys a large deviation principle (\abbr{LDP}) in $(\widetilde\cW_0,\delta_\square)$ with the rate function $I_p(\cdot)$. 
Further denote $\|W\|_1 = \int |W(x,y)|\, dxdy$, and considering the restricted spaces
\[  \cW_0^{(p)} := \big\{ W\in \cW_0 \,:\; \|W\|_1 = p\big\}\quad\mbox{ and }\quad\widetilde\cW_0^{(p)}=\big\{ W \in \widetilde \cW_0 \,:\; \|W\|_1 = p \big\} \,,
\]
here we deduce the analogous statement for the random graph $\cG(n,m)$, the uniform distribution over all graphs with $n$ vertices and exactly $m$ edges, with a rate function $J_p(\cdot)$ restricted to~$\widetilde\cW_0^{(p)}$. As we later conclude, 
the variational formulas of this \abbr{LDP} for 
$\cG(n,m)$, addressing such random graph structure conditioned on a large deviation, coincide with those studied earlier by 
Kenyon et al.\ and Radin et al.\ (cf.,~\cite{KRRS17a,KRRS17b,KRRS18}).

\begin{theorem}\label{thm:ldp}
Fix $0<p<1$ and let $m_n \in \N$ be such that $m_n/\binom{n}2\to p$ as $n\to\infty$.  Let $G_n \sim \cG(n,m_n)$. Then the sequence $\P(G_n\in\cdot)$  obeys 
the \abbr{LDP} in the space $(\widetilde \cW_0, \delta_\square)$ with 
the good rate function $J_p$, where $J_p(W)= I_p(W)$ if $W\in\widetilde\cW_0^{(p)}$ and is $\infty$ otherwise. That is, for any closed set $F \subseteq \widetilde \cW_0$,
\[
\limsup_{n \to \infty} n^{-2} \log \P(G_n\in
F) \leq - \inf_{W \in F} J_p(W)\,,
\]
and for any open $U \subseteq \widetilde \cW_0$,
\[
\liminf_{n \to \infty} n^{-2} \log \P(G_n\in
U) \geq - \inf_{W \in U} J_p(W)\,.
\]
\end{theorem}
Define
\begin{equation}
\label{eq:variational1}
\phi_H (p,r) := \inf\left\{I_p(W) \colon W \in \widetilde \cW_0~,~ t_H(W) \geq r\right\}
\end{equation}
and further let
\begin{equation}
\label{eq:variational2}
\psi_H (p,r) := \inf\left\{I_p(W) \colon W \in \widetilde \cW_0^{(p)} ~,~ t_H(W) \geq r\right\}
\end{equation}
(with $I_p$ having compact level sets in $(\widetilde \cW_0,\delta_\square)$ 
and $t_H(\cdot)$ continuous on $(\widetilde\cW_0,\delta_\square)$,
the infimums in~\eqref{eq:variational1},\eqref{eq:variational2}
are attained whenever the relevant set of graphons is nonempty).
For any $r \ge t_H(p)$ we relate the equivalent form of~\eqref{eq:variational2}
(see Corollary \ref{cor:variational}), given by
\begin{equation}\label{eq:variational-eq-r}
\psi_H(p,r)=\inf\left\{I_p(W) \colon W \in \widetilde \cW_0^{(p)} ~,~ t_H(W) = r\right\}\,,
\end{equation}
to the following 
variational problem that has been extensively studied (e.g.,~\cite{RS13,KRRS17a,KRRS17b,KRRS18}) in constrained random graphs such as the edge/triangle model (where $H$ is a triangle):
\begin{equation}
\label{eq:variational2'}
F_H (p,r) := \sup\left\{h_e(W) \colon W \in \widetilde \cW_0^{(p)}~,~ t_H(W) = r\right\}\,,
\end{equation}
where $h_e(x) = -\frac12(x\log x + (1-x)\log(1-x))$ is the (natural base) entropy function.
As $I_{p}(x)=-h_e(x)-\frac{x}2\log p -\frac{1-x}2\log(1-p)$ and
$\|W\|_1=p$ throughout $\widetilde\cW_0^{(p)}$,
we see that both 
variational problems for $F_H$ and $-\psi_{H}$ 
have the same set of optimizers, and
\[ F_H(p,r) = -\psi_H(p,r) + h_e(p)\,.\]

As a main application of their \abbr{LDP}, Chatterjee and Varadhan~\cite{CV11} showed that the large deviation rate function for subgraph counts in $\cG(n,p)$ for any fixed $0<p<1$ and graph $H$ reduces to the variational problem~\eqref{eq:variational1}. Namely, if $G_n\sim\cG(n,p)$ then
\[
\lim_{n \to \infty} n^{-2} \log \P\left(t_H(G_n)
  \geq r\right) = - \phi_H(p,r) \quad\mbox{ for every fixed $p,r\in(0,1)$ and $H$}\,,
\]
and, on the event $\{t_H(G_n)\geq r\}$, the graph $G_n$ is typically close to a minimizer of~\eqref{eq:variational1}.
Theorem~\ref{thm:ldp} implies the analogous statement for the random graph $\cG(n,m_n)$ w.r.t.\ the variational problem~\eqref{eq:variational2} (similar statements hold for lower tails of subgraph counts both in case of 
$\cG(n,p)$ and that of $\cG(n,m_n)$).

\begin{corollary}
  \label{cor:variational}
Fixing a subgraph $H$ and $0<p<1$, let
$r_H \in (t_H(p),1]$ denote the largest $r$ for which the 
collection of graphons in \eqref{eq:variational2} is nonempty.
\begin{enumerate}[(a)]
\item The \abbr{lsc} 
function $r \mapsto \psi_H(p,r)$ is 
zero on $[0,t_H(p)]$ 
and 
finite, 
strictly increasing 
on $[t_H(p),r_H]$. The nonempty set $F_\star$ of minimizers of \eqref{eq:variational2} is a single point $W_\star \equiv p$ for 
$r\leq t_H(p)$ and $F_\star$ coincides
for any $r \in [t_H(p),r_H]$ with the 
 minimizers of~\eqref{eq:variational-eq-r}.
\item For any $m_n \in \N$ such that $m_n/\binom{n}2\to p$ as $n\to\infty$ and
any right-continuity point $r \in [0,r_H)$ of
$t \mapsto \psi_H(p,t)$, the random graph $G_n\sim \cG(n,m_n)$ satisfies
\begin{align}\label{eq:asymptotic_rate_func}
\lim_{n \to \infty} n^{-2} \log \P\left(t_H(G_n)
  \geq r\right) &= - \psi_H(p,r) \,.
\end{align}
\item For any $(p,r)$ as in part (b),
and every $\epsilon >
0$ there is $C = C(H, \epsilon, p, r)>0$ so
that for all $n$ large enough
  \begin{equation}\label{eq:optimizers}
  \P\left(\delta_\square (G_n, F_\star) \ge \epsilon \;\big|\; t_H(G_n)
    \geq r\right) \leq e^{-C 
    n^2}\,.
  \end{equation}
\end{enumerate}
\end{corollary}
\begin{remark}\label{rem:no-limit}
Since the function $r\mapsto\psi_H(p,r)$ is monotone, it is continuous a.e.; however, the identity~\eqref{eq:asymptotic_rate_func} may fail 
when $\psi_H(p,\cdot)$ is discontinuous at $r$. For example, 
at $r=r_H$ the \abbr{lhs} of \eqref{eq:asymptotic_rate_func} 
equals $-\infty$ whenever $m_n/\binom{n}{2} \uparrow p$ slowly enough.
\end{remark}
\begin{remark}\label{rem:sparse}
The analog of~\eqref{eq:asymptotic_rate_func} in the sparse regime (with edge density $p_n =o(1)$) has been established in~\cite{CD16} in terms of a discrete variational problem in lieu of~\eqref{eq:variational1}, valid when $n^{-c_H}\ll p_n \ll 1$ for some $c_H>0$ (see also~\cite{Eldan16}, improving the range of~$p_n$, and~\cite{LZ-dense,LZ-sparse,BGLZ17,Zhao17} for analyses of these variational problems in the sparse/dense regimes).
In contrast with the delicate regime  $p_n=n^{-c}$, such results in the range $p_n\gg (\log n)^{-c}$ of $\cG(n,p)$ are a  straightforward consequence of the weak regularity lemma (cf.~\cite[\S5]{LZ-sparse}), and further 
extend to $\cG(n,m_n)$, where the discrete variational problem features an extra constraint on the number of edges (see
Proposition \ref{prop:sparse-var}).
\end{remark}

Consider $(p,r)$ in the setting of Corollary~\ref{cor:variational}. 
The studies of the variational problem for $F_H$ given in~\eqref{eq:variational2'} were motivated by the question of estimating the number of graphs with prescribed edge and $H$-densities, via the following relation:
\[
F_{H}(p,r) = \lim_{\delta\downarrow 0}\lim_{n\to\infty} \frac1{n^2}\log |\mathscr{H}_{n,p,r}^\delta|\mbox{ where }\mathscr{H}_{n,p,r}^\delta =  \left\{ G_n \,:\,  \begin{array}{l}\left| |E(G_n)|/\tbinom{n}{2}-p\right| \leq \delta\,,\\ \noalign{\medskip}|t_{H}(G_n) - r|\leq \delta\end{array}\right\}\,.\]
(This follows by general principles from the \abbr{LDP} of~\cite{CV11} 
for $\cG(n,p)$; see Proposition \ref{prop:exp-g}(a), or 
\cite[Thm~3.1]{RS13} for the derivation in the special case of the edge/triangle~model).
Corollary~\ref{cor:variational} allows us, roughly speaking, to interchange the order of these two limits; for instance, for any 
right-continuity point $r \ge t_H(p)$ of $t\mapsto \psi_H(p,t)$  
(which holds a.e.),
the same variational problem in~\eqref{eq:variational2'} also satisfies
\begin{equation}\label{eq:var2'-no-delta} F_{H}(p,r) = \lim_{n\to\infty} \frac{1}{n^2} \log |\mathscr{H}_{n,m_n,r}|\mbox{ where }
\mathscr{H}_{n,m,r} =  \left\{ G_n \,:\,  \begin{array}{l} |E(G_n)|=m\,,\\ \noalign{\medskip}t_{H}(G_n) \geq r\end{array}\right\}\,.
\end{equation}
(Indeed, $-\psi_{H}(p,r) = \lim_{n\to\infty} n^{-2}\log\P\left(t_{H}(\cG(n,m_n))\geq r\right)$, and  this log-probability is then translated to $\log|\mathscr{H}_{n,m_n,r}|$ by adding $n^{-2}\log \binom{\binom{n}2}{m_n} \to h_e(p) = F_{H}(p,r)+\psi_{H}(p,r)$.)
For the various results (as well as numerical simulations for the many problems related to~\eqref{eq:variational2'} that remain open), the reader is referred to~\cite{KRRS17a,KRRS17b,KRRS18} and the references therein. 

\medskip
Recall that the law of $\cG(n,m_n)$ can be represented as that of a random 
graph $G_n$ from the model $\cG(n,p)$, conditional on $|E(G_n)|=m_n$. 
While our choice of $m_n$ in Theorem~\ref{thm:ldp} is rather typical for 
$\cG(n,p)$ when $n \gg 1$, any \abbr{LDP} and in particular
the \abbr{LDP} of \cite{CV11}, deals only with open and closed sets. The 
challenge in deriving Theorem~\ref{thm:ldp} is thus in handling the 
point conditioning. 
To this end, we provide in Section~\ref{sec:cond_ldp}  
a general result (Proposition~\ref{prop:exp-g}) for deriving a conditional \abbr{LDP}, which we then combine in \S\ref{subsec:coupling} 
with a combinatorial coupling, and thereby prove 
Theorem~\ref{thm:ldp}. Building on the latter, 
\S\ref{subsec:thm2-proof} provides the proof
of Corollary~\ref{cor:variational}, whereas \S\ref{subsec:sparse} is 
devoted to the analog of \eqref{eq:asymptotic_rate_func} for $\cG(n,m_n)$
in the range $m_n \gg n^2 (\log n)^{-c_H}$
(see Proposition~\ref{prop:sparse-var}).

\section{Conditional LDP}\label{sec:cond_ldp}
The \abbr{LDP} for $\cG(n,m)$ is obtained by the next 
result, whose proof mimics that of~\cite[Theorem~4.2.16]{DZ} about 
exponential approximations (see~\cite[Definition~4.2.14]{DZ}). 
\begin{proposition}\label{prop:exp-g}
Suppose 
Borel probability measures $\{\mu_n\}$ on a metric space $(\cX,d)$ 
satisfy the \abbr{LDP} with rate $a_n \to 0$ and 
good rate function $I(\cdot)$. Fix a metric space~$(\sss,\rho)$,
a continuous map $f:(\cX,d) \to (\sss,\rho)$ and $s \in \sss$. For 
every $\eta>0$, let  
$Z_n^\eta$ denote radnom variables of the law 
\begin{equation}\label{eq:nu-def}
\nu_n^\eta := \mu_n\left( \cdot \mid \cB_{f,s,\eta}^o\right)\,,
\end{equation}
where
\begin{align*}
\cB_{f,s,\eta}&:=\left\{ x\in \cX : \rho(s,f(x))\leq\eta \right\}\,,& \cB^o_{f,s,\eta}&:=\left\{ x\in \cX : \rho(s,f(x))< \eta\right\}\,.
\end{align*}
\begin{enumerate}[(a)]\item If
\begin{equation}\label{eq-p-likely}
\lim_{n\to\infty} a_n \log \mu_n(\cB^o_{f,s,\eta}) = 0\qquad\mbox{for every $\eta>0$ fixed}\,,
\end{equation}
then for the good rate function 
\begin{align*}
J_0(x) &:= \begin{cases}
 I(x), & f(x)=s\, \\
 \infty, & \mbox{otherwise\,}	
 \end{cases}
\end{align*}
and any open $U\subset\cX$ and closed $F \subset \cX$, 
\begin{align}
\liminf_{\eta \to 0} \liminf_{n\to\infty} a_n \log \nu_n^\eta(U) &\geq 
-\inf_{x\in U} J_0(x)\,,\label{eq:lbd-zero}
\\	
\label{eq:ubd-zero}
\limsup_{\eta \to 0}
\limsup_{n\to\infty} a_n \log\nu_n^\eta(F) &\leq -\inf_{x\in F} J_0(x)\,.
\end{align}

\item Suppose~\eqref{eq-p-likely} holds and that
$\{Z_n^\eta\}$ form an exponentially good approximation of 
variables $Z_n \sim \nu_n$; i.e., for any $\delta>0$, there exist  
couplings $\P_{n,\eta}$ of $(Z_n,Z_n^\eta)$ so that 
\begin{equation}\label{eq:exp-g-apx}
\lim_{\eta \downarrow 0} \limsup_{n \to \infty}
a_n \log \P_{n,\eta} (d(Z_n,Z_n^\eta) > \delta) = - \infty \,.
\end{equation}
Then $\{\nu_n\}$ 
satisfy the \abbr{LDP} with rate $a_n \to 0$ and the good rate function 
$J_0(\cdot)$.
\end{enumerate}
\end{proposition}
\begin{proof}
We first deduce from \eqref{eq-p-likely} that for  
every $\eta>0$, open $U \subset \cX$ and 
closed $F \subset \cX$,
\begin{align}
 \liminf_{n\to\infty} a_n \log \nu_n^\eta(U) &\geq -\inf_{x\in U} J^o_\eta(x)\,,\label{eq:lbd-delta} 
\\
\label{eq:ubd-delta}
\limsup_{n\to\infty} a_n \log\nu_n^\eta(F) &\leq -\inf_{x\in F} J_\eta(x)\,,
\end{align}
where
\[
J_\eta(x) := \begin{cases}
 I(x), & x\in \cB_{f,s,\eta}\, \\
 \infty, & \mbox{otherwise}	
 \end{cases}
\,,\qquad J^o_\eta(x) := \begin{cases} I(x), & x\in \cB^o_{f,s,\eta}\, \\
\infty, & \mbox{otherwise.}	
\end{cases}
\]
Indeed, for any Borel set $A$ and $\eta>0$, 
\[ 
\mu_n(A \cap \cB^o_{f,s,\eta}) \le 
\nu_n^\eta(A) 
\leq
\frac{\mu_n(A\cap \cB_{f,s,\eta})}{\mu_n(\cB^o_{f,s,\eta})}\,.\]
Hence, for any open set $U$, we deduce from the \abbr{LDP} 
for $\{\mu_n\}$ that
\[ 
\liminf_{n\to\infty} a_n \log\nu_n^\eta(U) \geq 
\liminf_{n\to\infty} a_n \log \mu_n(U\cap \cB^o_{f,s,\eta}) \geq 
-\inf_{x\in U\cap \cB^o_{f,s,\eta}} I(x) = -\inf_{x\in U}J^o_\eta(x)\,.
\]
Similarly, for any closed set $F$ it follows from \eqref{eq-p-likely} that
\begin{align*}
 \limsup_{n\to\infty} a_n \log\nu_n^\eta(F) &\leq \limsup_{n\to\infty}
a_n \log \mu_n(F\cap \cB_{f,s,\eta}) \leq -\inf_{x\in F\cap \cB_{f,s,\eta}} I(x) = -\inf_{x\in F} J_\eta(x)\,.
\end{align*}
(a). In the lower bound \eqref{eq:lbd-delta} one obviously can 
use $J_0(\cdot) \ge J_\eta^o(\cdot)$, 
yielding \eqref{eq:lbd-zero}. Moreover, we get the 
bound 
\eqref{eq:ubd-zero} out of \eqref{eq:ubd-delta}, upon showing 
that for any closed $F \subseteq \cX$, 
\begin{equation}\label{eq:dz-4218}
\inf_{y \in F} \{J_0(y)\} \le 
\liminf_{\eta\downarrow 0} \inf_{y \in F} \{J_\eta(y)\} := \alpha \,.
\end{equation}
To this end, it suffices to consider only 
$\alpha<\infty$, in which case 
$J_{\eta_\ell}(y_\ell) \le \alpha + \ell^{-1}$
for some $\eta_\ell \downarrow 0$ and $y_\ell \in F$.
As $\{y_\ell\}$ is contained in the 
compact level set $\{ x : I(x) \le \alpha+1\}$, it has a limit
point $y_\star \in F$. Since $J_{\eta_\ell}(y_\ell) = I(y_\ell) \to \alpha$
it follows from the \abbr{lsc} of $x \mapsto I(x)$ 
that $I(y_\star) \le \alpha$. Passing to the convergent sub-sequence
$\rho(f(y_\ell),f(y_\star)) \to 0$. Further, recall that 
$\rho(s,f(y_\ell)) \le \eta_\ell \downarrow 0$, hence by the triangle inequality
$\rho(s,f(y_\star))=0$. Consequently, $J_0(y_\star) = I(y_\star) \le \alpha$ yielding \eqref{eq:dz-4218} and completing the proof of 
part (a).

\smallskip
\noindent
(b). Clearly, $J_\eta$ is a good rate function (namely, of compact
level sets $\{x : J_\eta(x) \leq \alpha\} = \{x: I(x) \leq \alpha\} 
\cap \cB_{f,s,\eta}$), and $J_\eta \le J^o_\eta \uparrow J_0$. If 
$J^o_\eta \equiv J_\eta$ then \eqref{eq:ubd-delta}--\eqref{eq:lbd-delta}
form the \abbr{LDP} for $\{\nu_n^\eta\}$ with the good rate function $J_\eta$.
While in general this may not be the case, assuming hereafter that 
\eqref{eq:exp-g-apx} holds and proceeding as in \cite[(4.2.20)]{DZ}, 
we get from \eqref{eq:lbd-delta} that $\{\nu_n\}$ 
satisfies the \abbr{LDP} lower bound with the rate function
\[
\underline{J}(y) := \sup_{\delta > 0} \liminf_{\eta \downarrow 0} 
\inf_{z \in B_{y,\delta}} \{ J^o_\eta(z) \} 
\,,\]
where $B_{y,\delta}=\{z \in\cX : d(y,z)<\delta\}$ (see \cite[(4.2.17)]{DZ}, noting that no \abbr{ldp} upper bound for $\nu_n^\eta$ is needed here). 
Since $y \in B_{y,\delta}$ for any $\delta>0$, we have that 
\[
J_0(y) = \lim_{\eta \downarrow 0} J^o_\eta(y) \ge \underline{J}(y) 
\]
and consequently $\{\nu_n\}$ trivially satisfies the \abbr{LDP} lower 
bound also with respect to the good rate function $J_0$. Now, 
precisely as in the proof of \cite[Theorem 4.2.16(b)]{DZ}, we get from 
\eqref{eq:exp-g-apx} and \eqref{eq:ubd-delta} that the corresponding  
\abbr{LDP} upper bound holds for $\{\nu_n\}$, thanks to \eqref{eq:dz-4218}
(see \cite[(4.2.18)]{DZ}), thereby completing the proof of
part (b) of Prop.~\ref{prop:exp-g}.
\end{proof}

\section{LDP for the uniform random graph}\label{sec:ldp-uniform}

\subsection{Proof of Theorem~\ref{thm:ldp}}\label{subsec:coupling}

Let $\mu_n$ be the law of $\cG(n,p)$, which obeys the \abbr{LDP} with good rate function $I_p(\cdot)$ on $(\widetilde\cW_0,\delta_\square)$ and speed $n^2$, and let $\nu_n$ denote the law of $\cG(n,m_n)$. We shall apply Proposition~\ref{prop:exp-g}(b)
for $\sss=\R$ and $s=p$, with $f$ denoting 
the $L^1$-norm on graphons (edge density):
\[ f(W) := \|W\|_1 = \iint W(x,y)\ dxdy\,. \]
With these choices, the role of $Z_n$ will be assumed by $G_n\sim \cG(n,m_n)$, whereas those of the random variables $Z_n^\eta$ will be assumed by the binomial random graph $\cG(n,p)$ conditioned on having between $\frac{1}{2} (p-\eta) n^2$ and 
$\frac{1}{2} (p+\eta) n^2$ edges:
\begin{equation}\label{eq:G_n_eta_def} G_n^\eta \sim \left( \cG(n,p) \;\big|\; B^o_{p,\eta}\right)\,,\quad\mbox{where}\quad B^o_{p,\eta} = \left\{ G : 
\tfrac{2|E(G)|}{n^2} \in (p-\eta,p+\eta)\right\}
\end{equation}
Note that $p_n:=2 m_n/n^2 \in (p-\eta,p+\eta)$ for all $n \ge n_0(\eta)$.
We couple $(G_n,G_n^\eta)$ so that for such $n$, detereministically,
\begin{equation}\label{eq:coupling-xor}
|E(G_n) \xor E(G_n^\eta)|< \eta n^2
\end{equation} 
(here $S\xor T$ denotes symmetric difference).
This is achieved by the following procedure:
 \begin{enumerate}[(i)]
 \item Draw $G_n \sim \cG(n,m_n)$.
 \item Independently of $G_n$ draw $E_n \sim \Bin(\binom{n}2,p)$ 
 and $M_n \sim \left( E_n \mid  |2 E_n/n^2 -p| < \eta \right)$.
 Let $D_n = M_n-m_n$ and obtain $G_n^\eta$ from $G_n$ as follows:
	\begin{itemize}[$\bullet$]
 \item{}[shortage] if $D_n\geq 0$: add a uniformly chosen subset of $D_n$ edges missing~from~$G_n$.
 		\item{} [surplus] if $D_n\leq 0$: delete a uniformly chosen subset of $D_n$ edges from $G_n$.
	\end{itemize}
 \end{enumerate}
Since $|D_n| < \eta n^2$ this guarantees~\eqref{eq:coupling-xor} and has $G_n\sim\nu_n$; the additional fact that $G_n^\eta\sim \nu_n^\eta$ is seen by noting that, if $G\sim \cG(n,p)$ then $|E(G)|\sim \Bin(\binom{n}2,p)$, and on the event that $G$ has $M$ edges, these are uniformly distributed (i.e., the conditional distribution is $\cG(n,M)$). 

We proceed to show that such $\{G_n^\eta\}$ form an exponentially good approximation of $G_n$. Indeed, from the 
identity $\prod_{i=1}^t a_i - \prod_{i=1}^t b_i = \sum_{j=1}^t (\prod_{i<j} a_i) (a_j-b_j)(\prod_{i>j} b_i)$ and the 
definition of $t_H(\cdot)$, we find that 
for any $H$ of $t$ edges and graphs $G,G'$ 
on $n$ vertices, 
\begin{equation}
	\label{eq:tH-bound}
	\left|t_H(G)-t_H(G')\right| \leq t \|W_G-W_{G'}\|_1 = \frac{2t}{n^2}   
	\left|E(G)\xor E(G')\right|
\end{equation} 
(see also~\cite[Lemma 10.22]{Lovasz12}).

Next, fixing $\delta>0$, we set $k(\delta) \in \N$ large enough so that
$\delta>22 /\sqrt{\log_2 k} $ (for example, $k = \big\lceil 2^{(25/\delta)^2}\big\rceil$), and recall the following result:
\begin{theorem*}[{\cite[Thm.~2.7(b)]{BCLSV08}}]
If $k\geq 1$ and the graphs $G,G'$ are such that 
for every simple graph $H$ on $k$ vertices
$|t_H(G)-t_H(G')|\leq 3^{-k^2}$, then 
$\delta_\square(W_G,W_{G'}) \leq 22/\sqrt{\log_2 k}$.
\end{theorem*}
To utilize this relation, set $ \eta_0(\delta) = k^{-2} 3^{-k^2}$ noting 
that if graphs $G,G'$ on $n$ vertices satisfy 
$|E(G)\xor E(G')| < \eta n^2$ for some $\eta\leq \eta_0$, then 
$\left|t_H(G)-t_H(G')\right| < 2 \binom{k}2 \eta_0 < 3^{-k^2}$ for every 
graph $H$ on $k$ vertices, and so by the preceding 
$\delta_\square(G,G') < \delta$. 
In particular, from \eqref{eq:coupling-xor} we deduce that 
for every $\eta\leq \eta_0$ and all $n \ge n_0(\eta)$, 
\[ \P\left(\delta_\square(G_n,G_n^\eta)>\delta\right) = 0\]
holds under the above coupling of $(G_n,G_n^\eta)$, thereby implying~\eqref{eq:exp-g-apx}.

Finally, Noting that $B^o_{p,\eta}$ of \eqref{eq:G_n_eta_def} is 
the event $|2E_n/n^2-p|<\eta$ (with $E_n \sim \Bin(\binom{n}2,p)$ under $\mu_n$),
we deduce from the \abbr{LLN} that $\mu_n(B^o_{p,\eta})\to 1$.
In particular, for any $\eta>0$ one has that 
$n^{-2}\log\mu_n( B^o_{p,\eta}) \to 0$, thereby verifying \eqref{eq-p-likely}
for the case at hand.
\qed

\subsection{Proof of Corollary~\ref{cor:variational}}\label{subsec:thm2-proof}

(a). Recalling that $J_p(W)=I_p(W)$ on $\widetilde \cW_0^{(p)}$ and 
otherwise $J_p(W)=\infty$, we express \eqref{eq:variational2} as
\[
\psi_H(p,r) = \inf_{W \in \Gamma_{\ge r}} \, \{ J_p(W) \} \,,
\]
for the closed set of graphons 
\begin{equation}\label{eq:tH-geq-r}
\Gamma_{\ge r} := \left\{ W \in \widetilde \cW_0 ~\,:\,~ t_H(W) \geq r \right\} \,,
\end{equation}
denoting by $\Gamma_{=r}$ the closed subset of graphons 
with $t_H(W)=r$.
The unique global minimizer of $J_p(\cdot)$ over $\widetilde \cW_0$
is $W_\star \equiv p$. With $W_\star \in \Gamma_{=t_H(p)}$, it follows
that $\psi_H(p,r)=0$ on $[0,t_H(p)]$. Next, for any
$r \in (t_H(p),r_H]$, the good rate function $J_p(\cdot)$ 
is finite on the nonempty set $\Gamma_{\ge r} \cap \widetilde \cW_0^{(p)}$,
hence $\psi_H(p,r)=\alpha$ is finite and positive, with the infimum 
in \eqref{eq:variational2} attained at the nonempty compact set
\begin{equation}\label{def:F-star}
F_\star = \Gamma_{\ge r} \cap \{ W \in \widetilde \cW_0 \,:\, 
J_p(W) \le \alpha \} \,.
\end{equation}
Fixing such $r$ and $W_r \in F_\star$, consider the 
map 
$W_r(\lambda) := \lambda W_r + (1-\lambda) W_\star$ from 
$[0,1]$ to $\widetilde \cW_0^{(p)}$. Thanks to the 
continuity of $\lambda \mapsto t_H(W_r(\lambda))$ on $[0,1]$, there
exists for any $r' \in [t_H(p),t_H(W_r))$ some $\lambda' =\lambda'(r') 
\in [0,1)$ such that $t_H(W_r(\lambda'))=r'$. Hence, due to  
the convexity of $J_p(\cdot)$,  
\[
\psi_H(p,r') \le J_p(W_r(\lambda')) \le \lambda' J_p(W_r)) = \lambda'\alpha < \alpha := \psi_H(p,r) \,.
\] 
We have shown that $\psi_H(p,r') < \psi_H(p,r)$ for all 
$r' \in [t_H(p),t_H(W_r))$. Recalling that $t_H(W_r) \ge r$,
it follows that $\psi_H(p,\cdot)$ is strictly increasing on 
$[t_H(p),r_H]$ and further, that necessarily $t_H(W_r)=r$ for 
any $W_r \in F_\star$. That is, the collection $F_\star$ of 
minimizers of \eqref{eq:variational2} then consists of only
the minimizers of \eqref{eq:variational-eq-r}. 

Next, if 
$\psi_H(p,r') \le \alpha < \infty$ for all $r'<r$ then 
there exist a pre-compact collection $\{W_{r'}, r'<r\}$ in
$(\delta_\square,\widetilde \cW_0)$, with $J_p(W_{r'}) \le \alpha$ 
and $t_H(W_{r'}) \ge r'$. By the continuity of $t_H(\cdot)$ and 
the \abbr{lsc} of $J_p(\cdot)$, it follows that 
$t_H(W_r) \ge r$ and $J_p(W_r) \le \alpha$ for any
limit point $W_r$ of $W_{r'}$ as $r' \uparrow r$. Consequently
$\psi_H(p,r) \le \alpha$ as well, establishing the 
stated left-continuity of $\psi_H(p,\cdot)$ on $[0,r_H]$. Finally, 
recall that an increasing function, finite on $[0,r_H]$ 
and infinite otherwise, is \abbr{lsc} iff it is left continuous
on $[0,r_H]$.

\noindent
(b). 
Considering the \abbr{LDP} bounds
of Theorem~\ref{thm:ldp} for the closed set $\Gamma_{\ge r}$
and its open subset $\Gamma_{>r} := \Gamma_{\ge r} \setminus \Gamma_{=r}$
we deduce that 
\begin{align*} 
- \lim_{r' \downarrow r} \{ \psi_H(p,r') \} =
- \inf_{W \in \Gamma_{>r}} \, \{ J_p(W) \} 
& \le 
\liminf_{n \to \infty} n^{-2} \log  \P\left(t_H(G_n) > r\right) \\
& \le 
\limsup_{n \to \infty} n^{-2} \log \P\left(t_H(G_n)
  \geq r\right) \le  - \psi_H(p,r) \,.
\end{align*}
By the assumed right-continuity of $t \mapsto \psi_H(p,t)$ 
at $r \in [0,r_H)$, the preceding inequalities must all hold 
with equality, resulting with \eqref{eq:asymptotic_rate_func}. 

\noindent
(c). Proceeding to prove \eqref{eq:optimizers}, we fix $(p,r)$ as in part (b).
Further fixing $\epsilon > 0$, let 
\[
B_{W',\epsilon} := \left\{ W \in \widetilde \cW_0 \,:\, 
\delta_\square(W,W') < \epsilon \right\}
\] 
denote open cut-metric balls and consider the closed subset of $\Gamma_{\ge r}$,
\begin{equation}\label{def:Gamma-r-eps}
\Gamma_{\ge r,\epsilon} := \Gamma_{\ge r} \bigcap_{W' \in F_\star} 
(B_{W',\epsilon})^c \,.
\end{equation}
In view of \eqref{eq:asymptotic_rate_func} and the fact that
\[
\{ \delta_\square (G_n, F_\star) \ge \epsilon, \; t_H(G_n) \geq r\}
=  \{ W_{G_n} \in \Gamma_{\ge r,\epsilon} \} \,,
\]  
it suffices for \eqref{eq:optimizers} to show that 
\[
\limsup_{n \to \infty} n^{-2} \log 
\P \left( W_{G_n} \in \Gamma_{\ge r,\epsilon} \right) < - \alpha \,.
\]
By the \abbr{ldp} upper-bound of Theorem~\ref{thm:ldp}, this in 
turn follows upon showing that  
\begin{equation}\label{eq:contr}
\inf_{W \in \Gamma_{\ge r, \epsilon}} \, \{J_p(W)\} \le \alpha  
\end{equation}
contradicts the definition of $F_\star$. Indeed, $J_p(\cdot)$ 
has compact level sets, so if~\eqref{eq:contr} holds
then $J_p(W_r) \le \alpha$ for some
$W_r \in \Gamma_{\ge r,\epsilon}$. Recall \eqref{def:F-star} 
that in particular $W_r \in F_\star$, hence \eqref{def:Gamma-r-eps} 
implies that $\delta_\square(W_r,W_r) \ge \epsilon>0$,  
yielding the desired contradiction.
\qed

\subsection{Sparse uniform random graphs}\label{subsec:sparse}
 In this section we show that, as was the case in $\cG(n,p)$, the analog of~\eqref{eq:asymptotic_rate_func}, giving the asymptotic rate function for $\cG(n,m)$ in the sparse regime $m_n= n^2/\log^c n$ for a suitably small $c>0$, can be derived in a straightforward manner from the weak regularity lemma. Indeed, 
the proof below follows essentially the same short argument used for 
$\cG(n,p)$ in~\cite[Prop.~5.1]{LZ-sparse}.

\begin{definition}[Discrete variational problem for upper tails]
Let $H$ be a graph with $\kappa$ edges, and let $b>1$. Denote the set of weighted undirected graphs on $n$ vertices by
\[ \widehat{\mathscr{G}}_{n} = \left\{ (a_{ij})_{1\leq i \leq j \leq n} : 0\leq a_{ij} \leq 1\,,\, a_{ij}=a_{ji}\,,\, a_{ii}=0 \mbox{ for all $i,j$}\right\}\,,\]
and extend the definition of the graphon $W_{\widehat G}$ in~\eqref{eq:W_G-def} to a weighted graph $\widehat G\in\widehat{\mathscr{G}}_n$ by replacing the weight 1 corresponding to an edge $(\lceil nx\rceil,\lceil ny \rceil)$ by the weight $a_{\lceil nx \rceil,\lceil ny\rceil}$.
Taking $ m_n \le \binom{n}2$ and $p_n=m_n/\binom{n}2$,
the variational problem for $G_n\sim \cG(n,m_n)$ is
\[\widehat{\psi}_H(n,m_n,b) := \inf\Big\{ I_{p_n}(W_{\widehat G}) \,:\; \widehat G\in\widehat{\mathscr{G}}_n\,,\,t_H(W_{\widehat G})\geq b\, p_n^{\kappa}\,,\,\sum_{ij}a_{ij}=p_n \Big\}\,.
\]
\end{definition}
\begin{remark*}
When $p_n\to p$ for some fixed $0<p<1$, and $r\in [p^\kappa,r_H]$ is a right-continuity point of $t\mapsto\psi_H(p,t)$ (whence~\eqref{eq:asymptotic_rate_func} holds), one has $ \psi_H(p, r) = \lim_{n\to\infty}\widehat\psi_H(n,m_n, r p^{-\kappa}) $ (e.g., rescale a sequence $\widehat G_n$ of minimizers for $\widehat\psi_H(n,m_n,rp^{-\kappa}+\epsilon)$ by $p/p_n$; conversely, for a minimizer $W$ for $\psi_H(p,r)$, one can take  a sequence $G_n$ with $W_{G_n}\to W$).
\end{remark*}

 \begin{proposition}\label{prop:sparse-var}
Fix $H$ be a graph with $\kappa$  edges, fix $b>1$ and for $m_n \in \N$
let $G_n\sim \cG(n,m)$ and $p_n  = m_n/\binom{n}2$. For every $\epsilon>0$ there exists some $K<\infty$ such that, if $p_n (\log n)^{1/(2\kappa)} \ge K$ and $n$ is sufficiently large then
 \begin{align*}
-\widehat\psi_H(n,m_n,b) - \epsilon \leq \frac1{n^2}\log \P(t_H(G_n)\geq b\, p_n^\kappa) \leq -\widehat\psi_H(n,m_n,b-\epsilon) + \epsilon\,.
 \end{align*}
 In particular, if $m_n\in \N$ is such that $p_n (\log n)^{1/(2\kappa)}\to\infty$ and  $\lim_{n\to\infty}\widehat\psi_H(n,m_n,t)$ exists and is continuous in some neighborhood of $t=b$, then
 \begin{align*}
\lim_{n\to\infty}\frac1{n^2}\log \P\left(t_H(G_n)\geq b\, p_n^\kappa\right) = -\lim_{n\to\infty} \widehat\psi_H(n,m_n,b)\,.
 \end{align*}
 \end{proposition}
 The following simple lemma, whose analog for upper tails in $\cG(n,p)$ (addressing only the event $\mathcal{E}_1$ below) was phrased in~\cite[Lemma~5.2]{LZ-sparse} for triangle counts in $\cG(n,p)$, is an immediate consequence of the independence of distinct edges and Cram\'er's Theorem.
\begin{lemma}\label{lem:log-prob-E1-E2}
Fix $\epsilon>0$ and suppose $n$ is sufficiently large. Let  $V_1,\ldots,V_s$ be a partition of $\{1,\ldots,n\}$, let  $\widehat{G}=(a_{ij})\in\widehat{\mathscr{G}}_s$ be such that $\sum a_{ij}=p=m/\binom{n}2$, and define
 \[ \mathcal{E}_1(G) = \bigcap_{\substack{i,j \\ a_{ij}>p}}\left\{d_{G}(V_i,V_j) \geq a_{ij}\right\} \,,\qquad\mathcal{E}_2(G)=\bigcap_{\substack{i,j \\ a_{ij}<p}}\left\{d_{G}(V_i,V_j) \leq a_{ij} \right\}\,,
 \]
where $d_G(X,Y)=\frac{\#\big\{(x,y)\in X\times Y: xy \in E(G)\big\}}{|X||Y|}$. Then  $G_n\sim\cG(n,m)$ has
 \begin{equation}
-I_p(W_{\widehat G}) - \epsilon \leq \frac1{n^{2}}\log\P\left(\mathcal{E}_1(G)\cap\mathcal{E}_2(G)\right)  \leq -I_p(W_{\widehat G}) + \epsilon\,.\end{equation}
 \end{lemma}
\begin{proof}Let $G'_n\sim \cG(n,p)$, and recall that $d_{G'_n}(V_i,V_j)|V_i||V_j|\sim \Bin(|V_i||V_j|,p)$ and $d_{G'_n}(V_i,V_i)\binom{|V_i|}2\sim\Bin(\binom{|V_i|}2,p)$, with these variables being mutually independent, thus
\begin{align*} \frac1{n^2}\log\P\left(\mathcal{E}_1(G'_n)\cap\mathcal{E}_2(G'_n)\right) &\leq -\frac1{n^2}\sum_{i<j}|V_i||V_j|I_p(a_{ij}) -\frac1{n^2} \sum_{i} \tbinom{|V_i|}2 I_p(a_{ii}) \\
&= -I_p(W_{\widehat G})+O(n^{-2}) \,.\end{align*}
Next, since $\P(G_n \in \cdot) = \P(G'_n \in \cdot\mid |E(G'_n)|=m)$, it follows that
\[ \left| \log\P\left(\mathcal{E}_1(G_n)\cap\mathcal{E}_2(G_n)\right) - \log\P\left(\mathcal{E}_1(G'_n)\cap\mathcal{E}_2(G'_n)\right)\right| \leq  -\log\P(\Bin(\tbinom{n}2,p)=m)\,.\] For $N=\binom{n}2$, by  definition $p=m/N$ and so $\P(\Bin(N,p)=m) \geq 1/\sqrt{2\pi p(1-p)N}$ provided that $N$ is large enough, and the result follows.
\end{proof}
Combining the weak regularity lemma (see, e.g.,~\cite[Lemma 9.3]{Lovasz12}) with the counting lemma for graphons (cf., e.g.,~\cite[Lemma 10.23]{Lovasz12}) implies the following.
\begin{lemma}\label{lem:counting-reg}
Let $\epsilon>0$ and set $M=4^{1/\epsilon^2}$. For every graph $G$ there is a partition $V_1,\ldots,V_s$ of its vertices, for some $s\leq M$, such that the weighted graph 
 $\widehat{G} \in \widehat{\mathscr{G}}_s$ in which $a_{ij} = d_G(V_i,V_j)$ satisfies that, for every graph $H$ with $\kappa$ edges, 
$\big|t_H(G)-t_H(\widehat{G})\big| \leq \kappa\epsilon$.
\end{lemma}

\begin{proof}[\textbf{\emph{Proof of Proposition~\ref{prop:sparse-var}}}]
By Lemma~\ref{lem:counting-reg}, if $G_n$ has $t_H(G_n)\geq b p_n^\kappa$ and $|E(G_n)|=m$ then there exists a partition  $V_1,\ldots,V_s$ of its vertices, for some $s\leq M$, such that the corresponding weighted graph $\widehat{G}$ satisfies $t_H(W_{\widehat{G}})\geq b p_n^\kappa -\kappa\epsilon$ and $\|W_{\widehat{G}}\| = p_n$ (note that the edge density is invariant under the partition). We may round each of the densities $a_{ij}$ of $\widehat G$ up to a multiple of $\epsilon$ (only increasing $t_H$), 
with the effect of potentially increasing the edge density to at most $p_n+\epsilon$. By rescaling we then arrive at $\widehat G'$ such that $\|W_{G'}\|_1=p_n$ and 
\[ t_H(W_{\widehat G'})\geq \frac{b p_n^\kappa -\kappa\epsilon}{(1+\epsilon)^\kappa} \geq b p_n^\kappa-\epsilon\]
provided that $\epsilon/ p_n^\kappa$ is small enough, which will indeed be the case by our assumption on $p_n$. Applying Lemma~\ref{lem:log-prob-E1-E2}, along with a union bound on the partition (at most $M^n$ possibilities) and the rounded $a_{ij}$'s (at most $(1/\epsilon)^{M^2}$ possibilities, the dominant factor), gives the required result, as the hypothesis that $p_n (\log n)^{1/2\kappa}$ is large enough guarantees that this union bound amounts to a multiplicative factor of at most $\exp(\epsilon' n^2)$.
\end{proof}

\subsection*{Acknowledgment} A.D.~was supported in part by NSF grant DMS-1613091 and E.L.~was supported in part by NSF grant DMS-1513403. 

\bibliographystyle{abbrv}
\bibliography{ldp_ref}

\end{document}